\documentclass[11pt]{article}

\usepackage[latin1]{inputenc}
\usepackage{t1enc}
\usepackage{amsmath, amssymb, amsfonts, amsthm, amsopn,epsfig}
\usepackage[none]{hyphenat}
\usepackage{tikz}
\usepackage{float}
\usepackage{graphicx}
\usepackage{caption}
\usepackage[margin=1in]{geometry}
\usepackage[toc,page]{appendix}
\usepackage{epstopdf}
\usepackage{etoolbox}
\usepackage[colorlinks=true]{hyperref}

\providecommand{\keywords}[1]{\textit{Keywords: } #1}

\floatstyle{plain}
\restylefloat{figure}
\usetikzlibrary{arrows}
\usetikzlibrary{shapes}

\tikzset{
  treenode/.style = {align=center, inner sep=0pt, text centered,
    font=\sffamily},
  arn_n/.style = {treenode, circle, black, font=\sffamily\bfseries, draw=black,
    fill=white, text width=1.5em},
  arn_x/.style = {treenode, rectangle, draw=black,
    minimum width=0.5em, minimum height=0.5em},
  arn_y/.style = {draw, shape=circle, fill=black}
}

\theoremstyle{plain}
\newtheorem{theorem}{Theorem}

\newtheorem{claim}[theorem]{Claim}

\newtheorem{lemma}[theorem]{Lemma}

\newtheorem{question}[theorem]{Question}

\theoremstyle{definition}

\makeatletter
\renewenvironment{proof}[1][\proofname]{\par
  \pushQED{\qed}%
  \normalfont \topsep6\p@\@plus6\p@\relax
  \trivlist
  \item[\hskip\labelsep
        \itshape
    #1\@addpunct{.}]\mbox{}\\*
}{
  \popQED\endtrivlist\@endpefalse
}
\makeatother

\makeatletter
\def\moverlay{\mathpalette\mov@rlay}
\def\mov@rlay#1#2{\leavevmode\vtop{%
   \baselineskip\z@skip \lineskiplimit-\maxdimen
   \ialign{\hfil$\m@th#1##$\hfil\cr#2\crcr}}}
\newcommand{\charfusion}[3][\mathord]{
    #1{\ifx#1\mathop\vphantom{#2}\fi
        \mathpalette\mov@rlay{#2\cr#3}
      }
    \ifx#1\mathop\expandafter\displaylimits\fi}
\makeatother

\makeatletter
\let\orgdescriptionlabel\descriptionlabel
\renewcommand*{\descriptionlabel}[1]{%
	\let\orglabel\label
	\let\label\@gobble
	\phantomsection
	\edef\@currentlabel{#1}%
	\let\label\orglabel
	\orgdescriptionlabel{#1}%
}
\makeatother

\newcommand{\cupdot}{\charfusion[\mathbin]{\cup}{\cdot}}

\DeclareMathOperator*{\C}{\mathcal{C}}
\DeclareMathOperator*{\G}{\mathcal{G}}
\DeclareMathOperator*{\B}{\mathcal{D}}

\title{The poset on connected graphs is Sperner}

\author{
  Stephen G.Z. Smith\thanks{Department of Mathematical Sciences, University of Memphis, Memphis, TN 38152, USA.
  \emph{email}: \textbf{sgsmith1@memphis.edu}}
  \and
  Istv\'{a}n Tomon\thanks{Department of Pure Mathematics and Mathematical Statistics, University of Cambridge, Wilberforce Road, Cambridge
   CB3\thinspace0WB, UK.
  \emph{email}: \textbf{i.tomon@cam.ac.uk}}
}

\begin{document}

\maketitle

\begin{abstract}
Let $\C$ be the set of all connected graphs on vertex set $[n]$. Then $\C$ is endowed with the following natural partial ordering: for
$G,H\in \C$, let $G\leq H$ if $G$ is a subgraph of $H$. The poset $(\C,\leq)$ is graded, each level containing the connected graphs with 
the same number of edges. We prove that $(\C,\leq)$ has the Sperner property, namely that the largest antichain of $(\C,\leq)$ is 
equal to its largest sized level. This answers a question of Katona.
\end{abstract}


\section{Introduction}

Let $(P,\leq)$ be a partially ordered set (poset). We only consider partially ordered sets with finitely many elements. A \emph{chain} in $P$ is 
a set $C\subset P$ of pairwise comparable elements. An \emph{antichain} $A\subset P$ is a set of pairwise incomparable elements. The poset 
$(P,\leq)$ is \emph{graded} if there exists a partition of $P$ into subsets $A_{0},...,A_{m}$ such that $A_{0}$ is the set of minimal elements
of $P$, and whenever $x\in A_{i}$ and $y\in A_{j}$ with $x<y$ and there is no $z\in P$ with $x<z<y$, then we have $j=i+1$. If such a 
partition exists, it is unique and the sets $A_{0},...,A_{m}$ are the \emph{levels} of $P$.

A graded poset $(P,\leq)$ is \emph{Sperner} if the largest antichain in $P$ is the largest sized level.

Let $m$ be a positive integer, $[m]=\{1,...,m\}$. The Boolean lattice $2^{[m]}$ is the power set of $[m]$ ordered by inclusion, and 
$[m]^{(k)}=\{A\subset [m]: |A|=k\}$. By the well known theorem of Sperner \cite{sperner}, the poset $(2^{[m]},\subset)$ is Sperner, the 
largest antichains being equal to $[m]^{(\lfloor m/2 \rfloor)}$ and $[m]^{(\lceil m/2\rceil)}$. The question whether certain posets are 
Sperner is widely studied. For a short list of such results, see \cite{anderson}. In this paper, we investigate the Sperner property of the following poset.

 Let $n$ be a positive integer and let $\C$ denote the set of all connected graphs on vertex set $[n]$. (In other words, $\C$ is the family of labeled connected graphs on $n$ vertices.) The family $\C$ is endowed with the following natural partial ordering: for 
 $G,H\in \C$, let $G\leq H$ if $G$ is a subgraph of $H$, or more formally, if $E(G)\subset E(H)$. When there is no risk of confusion, we shall simply write $\C$ when referring to the poset
 $(\C,\leq)$. Observe that $\C$ is graded, the levels of $\C$ being the families $\C^{(k)}=\{G\in \C: |E(G)|=k\}$ for $k=n-1,...,m$. The following question originates from Katona \cite{katona}.

 \begin{question}\label{question}
 Is $(\C,\leq)$ Sperner?
 \end{question}

 We prove that the answer is yes. More precisely, setting $m=\binom{n}{2}$ and  $M=\lceil m/2\rceil$, the main result of this paper is the following theorem.

\begin{theorem}\label{mainthm}
If $n$ is sufficiently large, the unique largest antichain in $\C$ is $\C^{(M)}$.
\end{theorem}

Let us make a remark about how this result compares to Sperner's theorem \cite{sperner}. Let $\G$ be the set of all graphs on vertex set $[n]$ and extend the ordering $\leq$ to $\G$ in the obvious way. Also, for $k=0,...,m$, let 
$\G^{(k)}$ be the set of graphs in $\G$ with $k$ edges. Observe that $(\G,<)$ is isomorphic to $(2^{[m]},\subset)$, hence $(\G,<)$ is Sperner. Note that $\C$ is a very dense subset of $\G$. As we shall see in Section \ref{sect:connected}, the size of $\C$ is at least $2^{m}(1-2^{-n-o(n)})$. This corresponds to the well known statement that a graph chosen uniformly at random among all graphs with $n$ vertices (that is an element of $G(n,1/2)$ in the Erd\H{o}s-R\'{e}nyi random graph model) is disconnected with a probability that is exponentially small.


\bigskip

A problem similar to Question \ref{question} has been considered in a paper of Jacobson, K\'{e}zdy and Seif \cite{induced}. Let $G$ be a
connected graph and let $(C(G),<)$ be the poset, whose elements are the connected, vertex-induced subgraphs of $G$, and $H<H'$ if $H$
is an induced subgraph of  $H'$. In \cite{induced}, it was proved that this poset need not be Sperner, even if $G$ is a tree.

\bigskip

This paper is organized as follows. In Section \ref{sect:prelim}, we discuss our notation and prove a few technical results. In Section \ref{sect:connected}, we shall prove various bounds on the number of connected graphs with certain properties. 
These bounds provide us with some of the ingredients needed for the proof of Theorem \ref{mainthm} in Section \ref{sect:matching}. In Section \ref{sect:problems}, we 
propose some open problems.

\section{Preliminaries}\label{sect:prelim}

Let us say a few words about our notation, which is mostly conventional. If $G$ is a graph, $V(G)$ is the vertex set of $G$, $E(G)$ is the set of its edges, and $e(G)=|E(G)|$. If $U\subset V(G)$, $G[U]$ denotes the subgraph of $G$ induced on the vertex set $U$. If $F\subset E(G)$, then $G-F$ is the 
graph on vertex set $V(G)$ and edge set $E(G)\setminus F$. If $e\in E(G)$, we simply write $G-e$ instead of $G-\{e\}$.

For the sake of readability, we use the notation $\textrm{exp}_2(x)= 2^x$,  when necessary. Furthermore, $\log$ denotes base $2$ logarithm. 

Our paper contains a lot of technical computations that are made more convenient by the following extension of the binomial coefficient. We define the binomial coefficient $\binom{x}{k}$ for any $k\in \mathbb{N}$, $x\in \mathbb{R}$ such that 

	$$\binom{x}{k}=\begin{cases} \frac{x(x-1)...(x-k+1)}{k!} &\mbox{if }  k \leq x,\\
	0 & \mbox{otherwise}. \end{cases}$$
	
	We collect some of the simple properties of $\binom{x}{k}$ in the following lemma.
	
\begin{lemma}\label{binom} Let $k\in \mathbb{N}$, $x\in \mathbb{R}$.
	\begin{description}
		\item[(i)\label{app1}] If $x\geq k$, we have $\binom{x}{k-1}/\binom{x}{k}=\frac{k}{x-k+1}$.
		\item[(ii)\label{app2}] Let $\delta$ be a non-negative integer and suppose that $k\leq x\leq 2k-\delta$. Then $\binom{x+\delta}{k}\geq2^{\delta}\binom{x}{k}$.
		\item[(iii)\label{app4}] $\binom{x}{k}\leq\binom{x+1}{k}$ and $\binom{x}{k}\leq\binom{x+1}{k+1}$.
	\end{description}
\end{lemma}

\begin{proof}
	\ref{app1} and \ref{app4} easily follows from the definition. 
	
	Now let us prove \ref{app2}. If we prove the case $\delta=1$, that is $\binom{x+1}{k}\geq 2\binom{x}{k}$ for $k\leq x\leq 2k-1$, the result follows by induction on $\delta$. But in this case, we have $\binom{x+1}{k}/\binom{x}{k}=(x+1)/(x-k+1)\geq 2$.
	
\end{proof}

 We remark that by continuity, for any fixed positive integer $k$ and a real number $r\geq 1$, there is a unique $x\in \mathbb{R}$ such that $r=\binom{x}{k}$.
 
 \bigskip

Throughout this paper, we shall also use the following simple inequalities.

\begin{lemma}\label{squares}
	Let $a_{1},...,a_{s}$ be positive integers and let $a_{1}+...+a_{s}=n$. We have
	\begin{align}\label{square1}\tag{i}
	\sum_{i=1}^{s}\binom{a_{i}}{2}\leq \binom{n-s+1}{2},
	\end{align}
	and
	\begin{align}\label{square2}\tag{ii}
	\sum_{1\leq i<j\leq s}a_{i}a_{j}\geq (n-s+1)(s-1)+\binom{s-1}{2}.
	\end{align}
	Also, if $a_{i}\leq k$ for $i\in [s]$, where $n/2<k\leq n-s+1$, then
	\begin{align}\label{square3}\tag{iii}
	\sum_{i=1}^{s}\binom{a_{i}}{2}\leq \binom{n-k-s+2}{2}+\binom{k}{2},
	\end{align}
	and
	\begin{align}\label{square4}\tag{iv}
	\sum_{1\leq i<j\leq s}a_{i}a_{j}\geq k(n-k).
	\end{align}
\end{lemma}

\begin{proof}
	 The function $f(x)=x^{2}$ is convex, so $\sum_{i=1}^{s}a_{i}^2$ attains its maximum under the conditions $\sum_{i=1}^{s}a_{i}=n$ and $a_{i}\in \mathbb{Z}^{+}$ when $a_{1}=...=a_{s-1}=1$ and $a_{s}=n-s+1$. Note that the left hand side of (\ref{square1}) is $\sum_{i=1}^{s}a_{i}^{2}/2-n/2$, and the left hand side of (\ref{square2}) is $(n^{2}-\sum_{i=1}^{s}a_{i}^{2})/2$, while the right hand sides of these inequalities are the respective values when $a_{1}=...=a_{s-1}=1$ and $a_{s}=n-s+1$.
	 
	 For the inequalities (\ref{square3}) and (\ref{square4}), notice that with the additional condition that $a_{i}\leq k$, $\sum_{i=1}^{s}a_{i}^2$ attains its maximum when $a_{1}=...=a_{s-2}=1$, $a_{s-1}=n-k-s+2$, $a_{s}=k.$ The right hand side of (\ref{square3}) is exactly $\sum_{i=1}^{s}\binom{a_{i}}{2}$ with these values inserted. On the other hand, we have 
	 	$$\sum_{1\leq i<j\leq s}a_{i}a_{j}\geq a_{s}(a_{1}+...+a_{s-1})=a_{s}(n-a_{s})=k(n-k),$$
	 which proves (\ref{square4}).
	 
\end{proof}

\section{Connectivity of graphs}\label{sect:connected}

In this section, we investigate the following problems. How many edges can a graph $G$ have, whose removal destroys the connectivity, or $2$-edge-connectivity of $G$? Also, what is the number of $2$-edge-connected graphs $G$ on vertex set $[n]$ in which there are exactly $r$ edges, whose removal destroys the $2$-edge-connectivity of $G$?

Let us start this section with the following well known result about the number of disconnected graphs. For completeness, we shall provide a short proof. A stronger form of this result can be found in  \cite{graphtheory}, p. 138 as well.

\begin{lemma}\label{disc}
The number of disconnected graphs on vertex set $[n]$ is less than $\textit{exp}_2\Bigl(\binom{n-1}{2}+o(n)\Bigr)$.
\end{lemma}

\begin{proof} A graph $G$ is disconnected if there is a partition of $[n]$ into two nonempty sets $A$ and $B$ such that there are no edges 
between $A$ and $B$. The number of disconnected graphs, where $|A|=1$ and $|B|=n-1$ is at most 
$n \cdot \mathrm{exp}_2\Bigl(\binom{n-1}{2}\Bigr)$, as we have $n$ choices for the partition $\{A,B\}$, and 
$\mathrm{exp}_2\Bigl(\binom{n-1}{2}\Bigr)$ number of different choices for the edges in $B$.

The number of disconnected graphs where $|A|,|B|\geq 2$ is at most $\mathrm{exp}_2\Bigl( n +\binom{n-2}{2}+1\Bigr)=\mathrm{exp}_2\Bigl( \binom{n-1}{2}+3\Bigr)$, as there are at most 
$\mathrm{exp}_2(n)$ number of choices for the partition $(A,B)$, and the number of ways to choose the edges inside $A$ and $B$ is at most 
$\mathrm{exp}_2\Bigl(\binom{|A|}{2}+\binom{|B|}{2}\Bigr) \leq \mathrm{exp}_2 \Bigl(\binom{n-2}{2}+1\Bigr)$.
Hence, the total number of disconnected graphs is at most $\mathrm{exp}_2 \Bigl( \binom{n-1}{2}+o(n) \Bigr)$.

\end{proof}

\bigskip

We define the \emph{block tree} of a connected graph $G$ as follows. An edge $e\in E(G)$ is a \emph{bridge}, if $G-e$ is disconnected. Let 
$B$ be the set of bridges in $G$ and let $A_{1},...,A_{t}$ be the vertex sets of the components of $G-B$. Then the block tree of $G$ is 
$Bt(G)=(B,\{A_{1},...,A_{t}\})$.

 \begin{figure}
 \centering
 \includegraphics[scale=1]{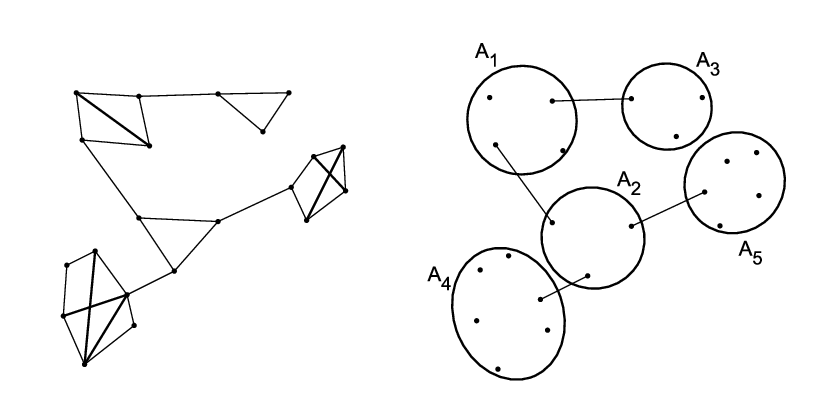}
 \caption{A graph and an illustration of its blocktree}
 \label{image1}
 \end{figure}

The following lemma lists the main properties of the block tree, which may be easily verified by the reader. 

\begin{lemma}\label{1conngraphs}
Let $G$ be a connected graph with block tree $(B,\{A_{1},...,A_{t}\})$. Then $|B|=t-1$ and $G[A_{i}]$ is $2$-edge-connected for $i\in [t]$.
\end{lemma}

If $G$ is a $2$-edge-connected graph, let $R(G)$ be the set of edges $f\in E(G)$ such that $G-f$ is not $2$-edge-connected. Lemma \ref{2connstruct} gives an upper bound on the size of $R(G)$.

\begin{lemma}\label{2connstruct}
Let $G$ be a $2$-edge-connected graph and let $H=G-R(G)$. Denote the number of components of $G-R(G)$ by $q$. Then $|R(G)|\leq 2q-2$.
\end{lemma}

To make our proof more convenient, we shall work with \emph{multi-graphs}. A multi-graph is a graph where we allow multiple edges between a pair of vertices, but no loops. We extend the definition of a \textit{cycle} as follows: a cycle is either $2$ vertices connected by $2$ edges or a simple graph that is a cycle. A \emph{chord} in a cycle $C$ is an edge not in $E(C)$ connecting two vertices of $C$. For example, if the vertices $x$ and $y$ are connected by $3$ edges, any two edges form a cycle and the third edge is a chord of this cycle.

\begin{proof}
Let the components of $H$ be $H_{1},...,H_{q}$. Every edge in $R(G)$ connects two different components in $H$. Define the multi-graph $K$ on vertex set $[q]$ as follows: if $H_{i}$ and $H_{j}$ are connected by $l$ edges in $G$, then $i$ and $j$ are connected by $l$ edges in $K$.

Note that the graph $K$ cannot contain a cycle with a chord. Otherwise, suppose that there is a cycle with vertices $i_{1},...,i_{s}$ and a chord $i_{a}i_{b}$. Let $e\in R(G)$ be an edge connecting $H_{i_{a}}$ and $H_{i_{b}}$ in $G$. Then $H_{i_{a}}$ and $H_{i_{b}}$ are still connected by at least 2 disjoint paths in $G-e$, hence $e$ cannot be an element of $R(G)$.

Our lemma follows from the following result about multi-graphs without cycles with a chord.

\begin{claim}\label{multigraph}
	If $L$ is a multi-graph on $q$ vertices without a cycle with a chord, then $e(L)\leq 2q-2$.
\end{claim}

\begin{proof} We proceed by induction on $q$. If $q=1$, $E(L)$ is empty, so we are done. Suppose that $q>1$. If $L$ has a vertex $v$ of degree at most $2$, then let $L'=L-v$. Then $L'$ has $q-1$ vertices, at least $e(L)-2$ edges, and does not contain a chorded cycle. Hence, by induction, $e(L)-2\leq 2q-4$, which gives $e(L)\leq 2q-2$. Now suppose that every vertex of $L$ has degree at least $3$. Let $v_{1},...,v_{s}$ be the consecutive vertices of a longest path in $L$. Every neighbor of $v_{1}$ is contained in the set $\{v_{2},...,v_{s}\}$, otherwise we can find a longer path in $L$. Hence, there exist $i,j$ satisfying $2\leq i\leq j\leq s$ such that the multi-set $E(L)$ contains three different edges, $v_{1}v_{2},v_{1}v_{i}$ and $v_{1}v_{j}$. But then $v_{1},...,v_{j}$ forms a cycle, and $v_{1}v_{i}$ is a chord of this cycle.
\end{proof}

\bigskip

As $K$ does not contain a cycle with a chord and has $|R(G)|$ edges, we get $|R(G)|\leq 2q-2$. This completes the proof of Lemma \ref{2connstruct}.

\end{proof}

\bigskip
We remark that if $R(G)$ is non-empty, it has at least $2$ elements. This is true because if $e\in R(G)$, then $G-e$ contains a bridge $f$. 
But then $f\in R(G)$ as well.

\bigskip

Let $I_{r}$ be the set of $2$-edge-connected graphs $G$ such that $|R(G)|=r$, and let $I_{r}^{(k)}=I_{r}\cap \C^{(k)}$. In the next lemma, we 
give an upper bound on the size of $I_{r}^{(k)}$. Recall that $M = \lceil \binom{n}{2} / 2 \rceil$.

\begin{lemma}\label{numof2conn}
Let $\epsilon$ be a positive real number. There exists $n_{1}(\epsilon)$ such that if $n>n_{1}(\epsilon)$, the following holds. For any 
positive integers $r$ and $k$ satisfying $2\leq r\leq n$ and $M\leq k\leq M+n$, we have
$$|I_{r}^{(k)}|\leq\binom{\binom{n-r/2}{2}+\epsilon rn}{k}.$$
\end{lemma}

\begin{proof}
Let $q=\lceil r/2 \rceil +1$.  If $G$ is a $2$-edge-connected graph with $|R(G)|=r$, then $G-R(G)$ has at least $q$ components by  Lemma \ref{2connstruct}. We now count the number of graphs $G$ where $G-R(G)$ has exactly $s$ components. Note that $s\leq r$, otherwise the edges of $R(G)$ could not connect all the components of $G-R(G)$.

 The number of graphs $G$, for which $|R(G)|=s$, and where the components in $G-R(G)$ have sizes $a_{1},...,a_{s}$ with $e_{1},...,e_{s}$ 
 edges inside them, respectively, is at most
\begin{equation}\label{2conngraphs}
\binom{n^{2}}{r}\binom{n}{a_{1},...,a_{s}}\prod_{i=1}^{s}\binom{\binom{a_{i}}{2}}{e_{i}}.
\end{equation}
Here, $\binom{n^{2}}{r}$ is an upper bound on the number of ways to pick the edges of $R(G)$, $\binom{n}{a_{1},...,a_{s}}$ is the number of 
ways to partition $[n]$ into parts of size $a_{1},...,a_{s}$, and $\binom{\binom{a_{i}}{2}}{e_{i}}$ is the number of ways to choose the 
$e_{i}$ edges in a component of size $a_{i}$.
We shall prove that (\ref{2conngraphs}) is at most $\binom{\binom{n-s+1}{2}}{k} \textrm{exp}_2(3\epsilon rn/6)$. Let us bound the terms 
in (\ref{2conngraphs}).

First, $\binom{n^{2}}{r} \leq \textrm{exp}_2(2r\log n) < \textrm{exp}_2(\epsilon rn/6)$, if $n$ is sufficiently large given $\epsilon$.

Also, $\binom{n}{a_{1},...,a_{s}}\leq s^{n}=  \textrm{exp}_2(n\log s)$. Unfortunately, if $r$ is small, we cannot bound this term by $
\textrm{exp}_2(c\epsilon rn)$, where $c$ is some fixed constant. We shall overcome this obstacle later in the proof.

Finally,
$$\prod_{i=1}^{s}\binom{\binom{a_{i}}{2}}{e_{i}}\leq \binom{\sum_{i=1}^{s}\binom{a_{i}}{2}}{k-r}\leq \binom{\binom{n-s+1}{2}}{k-r},$$
where the last inequality holds by (\ref{square1}) in Lemma \ref{squares}.
Here,
$$\binom{\binom{n-s+1}{2}}{k-r}<\binom{r+\binom{n-s+1}{2}}{k},$$
see \ref{app4} in Lemma \ref{binom}.
Hence, we have
$$\prod_{i=1}^{s}\binom{\binom{a_{i}}{2}}{e_{i}}\leq \binom{\binom{n-s+1}{2}+\epsilon rn/6}{k},$$
provided $n>6/\epsilon$.

First, suppose that $r$ is such that $\log r<\epsilon r/6$. In this case, we have $\binom{n}{a_{1},...,a_{s}}\leq \textrm{exp}_2(\epsilon rn/
6).$
 Hence, (\ref{2conngraphs}) is at most $\binom{\binom{n-s+1}{2}+\epsilon rn/6}{k} \cdot \textrm{exp}_2(2\epsilon rn/6)$.

Now consider the case when $\log r>\epsilon r/6$. Then $r<R(\epsilon)$, where $R(\epsilon)$ is a constant only depending on $\epsilon$. In 
this case, we shall bound the product
\begin{equation}\label{eq2}
\binom{n}{a_{1},...,a_{s}}\prod_{i=1}^{s}\binom{\binom{a_{i}}{2}}{e_{i}}.
 \end{equation}

 Without loss of generality, suppose that $a_{1}\geq...\geq a_{s}$ and observe that $\binom{n}{a_{1},...,a_{s}}<n^{a_{2}+...+a_{s}}$. Thus, 
 if $a_{1}\geq n-4r$, then $\binom{n}{a_{1},...,a_{s}}<n^{4r}<\textrm{exp}_2(\epsilon rn/6)$, if $n$ is sufficiently large given $\epsilon$. 
 Now suppose that $a_{1}<n-4r$.
Applying (\ref{square3}) in Lemma \ref{squares}, we get
$$\sum_{i=1}^{s}\binom{a_{i}}{2}\leq \binom{4r-s-2}{2}+\binom{n-4r}{2}$$
$$\leq 8r^{2}+\binom{n-4r}{2}.$$
Suppose $n>20R(\epsilon)$, then the inequality
$$8r^{2}+\binom{n-4r}{2}\leq \binom{n-s+1}{2}-2rn$$
holds as well. Hence,
$$\prod_{i=1}^{s}\binom{\binom{a_{i}}{2}}{e_{i}}<\binom{\sum_{i=1}^{s}\binom{a_{i}}{2}}{k-r}\leq \binom{\binom{n-s+1}{2}-2rn}{k-r}<
\binom{\binom{n-s+1}{2}-2rn+r}{k},$$
where the last inequality holds by \ref{app4} in Lemma \ref{binom}. Also, using \ref{app2} in Lemma \ref{binom},

$$\binom{\binom{n-s+1}{2}-2rn+r}{k}\leq \binom{\binom{n-s+1}{2}}{k} \textrm{exp}_2(-rn).$$

Thus, we can bound (\ref{eq2}) from above by $\binom{\binom{n-s+1}{2}}{k}$, and so (\ref{2conngraphs}) is at most
$$\binom{\binom{n-s+1}{2}+\epsilon rn/6}{k} \textrm{exp}_2(2\epsilon rn/6)$$ in this case as well.

\bigskip

Now let us bound the number of all $2$-edge-connected graphs with $k$ edges, for which $|R(G)|=r$ and $G-R(G)$ has $s$ components. The number 
of such graphs is at most
\begin{equation}\label{eqsomething}
\sum_{a_{1}+...+a_{s}=n}\sum_{e_{1}+...+e_{s}=k-r}\binom{n^{2}}{r}\binom{n}{a_{1},...,a_{s}}\prod_{i=1}^{s}\binom{\binom{a_{i}}{2}}{e_{i}}.
\end{equation}
The first sum has exactly $\binom{n}{s-1}$ terms since $a_i \geq 1$ for every $i \in [s]$, while the second sum has $\binom{k-r+s}{s-1}$ 
terms. Therefore, (\ref{eqsomething}) is at most
$$\binom{n}{s-1}\binom{k-r+s}{s-1}\binom{\binom{n-s+1}{2}+\epsilon rn/6}{k} \textrm{exp}_2(2\epsilon rn/6).$$
Here, $\binom{n}{s-1}\leq \textrm{exp}_2(r\log n)$ and $\binom{k-r+s}{s-1}< \textrm{exp}_2(2r\log n)$. Thus, (\ref{eqsomething}) is at most
$$\binom{\binom{n-s+1}{2}+\epsilon rn/6}{k} \textrm{exp}_2(3\epsilon rn/6),$$
provided $n$ is sufficiently large given $\epsilon$.

Finally, the number of $2$-edge-connected graphs with $|R(G)|=r$ and $k$ edges is at most

$$\sum_{i=q}^{r}\binom{\binom{n-i+1}{2}+\epsilon rn/6}{k} \mathrm{exp}_2(3\epsilon rn/6)<\binom{\binom{n-q+1}{2}+\epsilon rn/6}{k}
\mathrm{exp}_2(4\epsilon rn/6).$$

Applying \ref{app2} in Lemma \ref{binom}, we get
$$|I_{r}^{(k)}|\leq\binom{\binom{n-q+1}{2}+\epsilon rn}{k}.$$

\end{proof}

In the proof of Theorem \ref{mainthm}, we shall also use the following technical lemma. Again, recall that $M = \lceil \binom{n}{2} / 2 \rceil$.

\begin{lemma}\label{tech}
Let $n>150$. Let $G$ be a connected graph on vertex set $[n]$ such that $e(G)\geq M$ and $Bt(G)=(B,\{A_{1},...,A_{t}\})$. Suppose that $|A_i| 
\leq n-2$ for $i \in [t]$. Then
\begin{equation}\label{equlemma}
\sum_{1\leq i<j\leq t}|A_{i}||A_{j}|-2(t-1)-\sum_{i=1}^{t}|R(G[A_{i}])|\geq n.
\end{equation}
\end{lemma}

\begin{proof}
By Lemma \ref{2connstruct}, we have $|R(G[A_{i}])|< 2|A_{i}|$. Hence, $\sum_{i=1}^{t}|R(G[A_{i}])|< 2n$.

First, suppose that $\max\{|A_{1}|,..,|A_{t}|\}\leq n-6$. By (\ref{square4}) in Lemma \ref{squares}, we have
$$\sum_{1\leq i<j\leq t}|A_{i}||A_{j}|\geq 6(n-6)\geq 5n.$$
Hence, using the trivial bound $t-1<n$, we have that (\ref{equlemma}) holds.

Now suppose that $|A_{1}|\geq n-5$. In this case, we have $t\leq 6$. Let $H=G[A_{1}]$. Every edge of $G$ not contained in $H$ is either in $B
$ or it is an edge of $G[[n]\setminus A_{1}]$. Hence, the number of edges not contained in $H$ is at most $20$, so $e(H)\geq M-20$.

Let $H_{1},...,H_{q}$ be the vertex sets of the components of $H-R(H)$. Then, by Lemma \ref{2connstruct}, the number of edges of $H$ is at 
most
$$2q-2+\sum _{i=1}^{q}\binom{|V(H_{i})|}{2}<2n+\binom{n-q+1}{2},$$
where the inequality holds by (\ref{square1}) in Lemma \ref{squares}.
Comparing the lower and upper bounds on $e(H)$ we get the inequality
 $$M-20<2n+\binom{n-q+1}{2}.$$
If $q>n/3$, the right hand side of the inequality is at most $2n^{2}/9+3n$, while the left hand side is larger than $n^{2}/4-n$. This is a
contradiction, noting that $2n^{2}/9+3n<n^{2}/4-n$ for $n>150$. Hence, we have $q<n/3$, implying $|R(H)|<2n/3$. This gives
$$\sum_{i=1}^{t}|R(G[A_{i}])|\leq |R(H)|+2(|A_{2}|+...+|A_{t}|)<2n/3+10.$$
Since $|A_{1}|\leq n-2$, we have $\sum_{1\leq i<j\leq t}|A_{i}||A_{j}|\geq 2(n-2)$ by (\ref{square4}) in Lemma \ref{squares}, so (\ref{equlemma}) holds.

\end{proof}

\section{Matchings between levels}\label{sect:matching}

In this section, we prove Theorem \ref{mainthm}.

Let $n-1\leq k,l\leq m$. We say that there is a complete matching from $\C^{(k)}$ to $\C^{(l)}$, if there is an injection $f: \C^{(k)}
\rightarrow \C^{(l)}$ such that $G$ and $f(G)$ are comparable for all $G\in \C^{(k)}$. The next lemma states that to prove Theorem 
\ref{mainthm}, it is enough to find a complete matching from the smaller sized level to the larger sized level for any two consecutive 
levels. Due to its simplicity, we shall only sketch the proof of this lemma.

\begin{lemma}\label{trivial}
Suppose that there is a complete matching from $\C^{(k)}$ to $\C^{(k+1)}$ for $k=n-1,...,M-1$, and there is a complete matching from $\C^{(l+1)}$ to $\C^{(l)}$ for $l=M,...,m-1$. Then the largest antichain in $\C$ is $\C^{(M)}$.
\end{lemma}

\begin{proof}
Using the complete matchings, one can build a chain partition of $\C$ into $|\C^{(M)}|$ chains. But the size of the maximal 
antichain in $\C$ is at most the number of chains in any chain partition of $\C$.

\end{proof}

\bigskip

First, we show that if we are below the middle level $\C^{(M)}$, or at least $n$ above the middle level, then it is easy to prove the 
existence of a complete matching between consecutive levels.

Let $X\subset \C^{(k)}$ for some $n-1\leq k\leq m$. The \emph{lower shadow} of $X$ is
$$\Delta(X)=\{G\in \mathcal{C}^{(k-1)}: \exists H\in X, G<H\},$$
and the \emph{upper shadow} of $X$ is
$$\nabla(X)=\{G\in \mathcal{C}^{(k+1)}: \exists H\in X, H<G\}.$$

In our proofs, we shall apply the well known theorem of Hall \cite{hall}.

\begin{theorem}(Hall's theorem)
	Let $G=(A,B;E)$ be a bipartite graph. There is a complete 
	matching in $G$ from $A$ to $B$ if and only if $|X|\leq|\Gamma(X)|$ for all $X\subset A$, where $\Gamma(X)$ denotes the set of vertices adjacent to some element of $X$.
\end{theorem}  

First, let us deal with the levels below $\C^{(M)}$.

\begin{lemma}\label{smalllevels}
There is a complete matching from $\C^{(k)}$ to $\C^{(k+1)}$ for ${k=n-1,...,M-1}$.
\end{lemma}

\begin{proof} Let $X\subset \C^{(k)}$. By Hall's theorem, it is enough to show that $|X|\leq |\nabla(X)|$. Let $B$ be the bipartite graph 
with vertex partition ($X,\nabla(X)$), and the edges of $B$ being the comparable pairs.  If $G\in X$, the degree of $G$ is $m-k$. Also, if $H
\in \nabla(X)$, the degree of $H$ is at most $k+1$.

Let $e$ be the number of edges of $B$. Then, counting $e$ from $X$, and then from $\nabla(X)$, we have

$$|X|(m-k)=e,$$
and
$$e\leq|\nabla(X)|(k+1).$$
Hence,

$$|X|\leq|\nabla(X)|(k+1)/(m-k)\leq |\nabla(X)|.$$

\end{proof}

\bigskip

Using similar ideas, we now show that if we are above the middle level by at least $n$, then there is a matching from $\C^{(k+1)}$ to $
\C^{(k)}$.

\begin{lemma}\label{largelevels}
There is a complete matching from $\C^{(k+1)}$ to $\C^{(k)}$ for ${k=M+n,...,m}$.
\end{lemma}

\begin{proof}
Let $X\subset \C^{(k+1)}$. By Hall's theorem, it is enough to show that $|X|\leq |\Delta(X)|$. Let $B$ be the bipartite graph with vertex 
partition ($X,\Delta(X)$), and the edges of $B$ being the comparable pairs. If $G\in \Delta(X)$, then the degree of $G$ in $B$ is at most $m-
k$.

Now let $G\in X$. If $e\in E(G)$ such that $G-e$ is not an element of $\C$, then $e$ is a bridge of $G$. However, by Lemma \ref{1conngraphs},
the number of bridges of $G$ is at most $n-1$. Hence, the degree of $G$ is at least $k+2-n$. Counting the number of edges of $B$ two ways, we 
get
$$|X|(k+2-n)\leq |E(B)|,$$
and
$$|\Delta(X)|(m-k)\geq |E(B)|.$$
Hence,

$$\frac{|\Delta(X)|}{|X|}\geq \frac{k+2-n}{m-k}\geq 1.$$

\end{proof}

Proving that there is a matching from $\C^{(k)}$ to $\C^{(k-1)}$ for the values of $k$ that are slightly larger than $M$ is more difficult. 
The remainder of this section is devoted to this problem. Before showing the details, we briefly outline the strategy for showing that there 
exists a complete matching from $\C^{(k)}$ to $\C^{(k-1)}$, where $M+1 \leq k<M+n$.

Our goal is to show that for every $X\subset \C^{(k)}$, we have $|\Delta(X)|\geq |X|$. To accomplish this, we write $X$ as $Y\cup Z$, where $Y$ is the set of $2$-edge-connected graphs in $X$ and $Z$ is the set of the non-2-edge-connected graphs in $X$. We first 
show that if the two sets, $Y$ and $Z$, do \emph{not} have roughly the same size, then the larger of the two has a lower shadow that is already larger than $|X|$.

Now suppose that $|Y|\approx |Z|$. We show the existence of three functions $c_{1},c_{2},c_{3}:\mathbb{N}\rightarrow \mathbb{R}^{+}$ satisfying the following properties:
\begin{enumerate}
	\item  $|\Delta(Y)| \geq |Y|(1+c_{1}(|Y|))$,
	\item  $|\Delta(Z)| \geq |Z|(1+c_{2}(|Z|))$,
	\item if $U$ is the set of $2$-edge-connected graphs in $\Delta(Y)$, then $|U|\geq |Y|(1-c_{3}(|Y|))$,
	\item $c_{1}(|Y|)c_{2}(|Z|)\geq c_{3}(|Y|)$, if $|Y|\approx |Z|$.
\end{enumerate}
Roughly, 1. and 2. state that the lower shadow of $Y$ and $Z$ is slightly larger than $Y$ and $Z$, respectively. Now, we would like to guarantee that $\Delta(X)=\Delta(Y)\cup \Delta(Z)$ is also larger than $Y\cup Z$. If this is not the case, then we must have that $\Delta(Y)\setminus \Delta(Z)$ is too small. But note that as $\Delta(Z)$ contains only non-2-edge-connected graphs, $U$ is contained in $\Delta(Y)\setminus \Delta(Z)$. Hence, $|U|$ is a lower bound on the size of the set $\Delta(Y)\setminus \Delta(Z)$. Thus, 3. tells us that $\Delta(Y)\setminus \Delta(Z)$ cannot be much smaller than $Y$, and property 4. guarantees (as we shall see later) that we truly have $|Y\cup Z|\leq |\Delta(Y)\cup\Delta(Z)|$.

\bigskip

We remind the reader that $\G$ is the family of all graphs on vertex set $[n]$. For $X\subset \C^{(k)}$, let $\partial(X)=\{H\in \G^{(k-1)}:
\exists \, G\in X, H<G\}$. As $(\G,<)$ is isomorphic to $(2^{[m]},\subset)$, the Kruskal-Katona theorem \cite{katona1, kruskal} tells us 
which subfamily of $\G^{(k)}$ of given size minimizes the lower shadow. Instead of using this, however, we use a weaker form of the Kruskal-
Katona theorem, proved by Lov\'{a}sz \cite{lovaszlemma}. This affords us a computationally more convenient way to obtain a lower bound on the size of $\partial(X)$.

\begin{lemma}\label{lovasz} \textnormal{(Lov\'{a}sz \cite{lovaszlemma})}
Let $X\subset \C^{(k)}$ be nonempty and let $x$ be a real number such that $|X|=\binom{x}{k}$. Then
$$|\partial(X)|\geq \binom{x}{k-1}.$$
In particular,

  $$\frac{|\partial(X)|}{|X|}\geq \frac{k}{x-k+1}.$$
\end{lemma}

 We remind the reader that we use the extended definition of binomial coefficients introduced in Section \ref{sect:prelim}, so both in the previous lemma and in what comes, $x$ need not to be an integer in $\binom{x}{k}$.

Let $\B$ be the set of $2$-edge-connected graphs in $\C$ and let $\B^{(k)}=\C^{(k)}\cap \B$.  If $X\subset \B^{(k)}$, then $\Delta(X)=
\partial(X)$. Hence, we can use Lemma \ref{lovasz} to get a lower bound for the size of $\Delta(X)$.

\bigskip

In the next lemma we show that if the size of $X\in \C^{(k)}$ is sufficiently large, then we have $|\Delta(X)|\geq|X|$.

\begin{lemma}\label{largeX}
Let $\epsilon>0$. There exists $n_{2}(\epsilon)$ such that if $n>n_{2}(\epsilon)$ the following holds. Let $M+1\leq k<M+n$ and let $|X|=
\binom{x}{k}$, where $x>\binom{n-1}{2}+\epsilon n$. We have $|\Delta(X)|>|X|$.
\end{lemma}

\begin{proof} By Lemma \ref{lovasz},
$$|\partial(X)|\geq \binom{x}{k-1}.$$
Let $D$ be the set of disconnected graphs with $k-1$ edges. By Lemma \ref{disc}, $$|D|\leq \mathrm{exp}_2\left(\binom{n-1}{2}+o(n)\right).$$ Also,
$$|\Delta(X)|=|\partial(X)\setminus D|\geq|\partial(X)|-|D|\geq \binom{x}{k-1}- \mathrm{exp}_2\left(\binom{n-1}{2}+o(n)\right).$$
Thus, we get

$$|\Delta(X)|-|X|\geq \binom{x}{k-1}-\binom{x}{k}- \mathrm{exp}_2\left(\binom{n-1}{2}+o(n)\right)=$$
$$=\binom{x}{k-1}\frac{2k-x-1}{k}- \mathrm{exp}_2\left(\binom{n-1}{2}+o(n)\right)>\binom{\binom{n-1}{2}+\epsilon n}{k-1}\frac{1}{n^{2}}- \mathrm{exp}
_2\left(\binom{n-1}{2}+o(n)\right).$$

By \ref{app2} in Lemma \ref{binom}, we have $\binom{\binom{n-1}{2}+\epsilon n}{k-1}\geq \binom{\binom{n-1}{2}}{k-1} \cdot \mathrm{exp}
_2(\epsilon n)$. Also, $\binom{\binom{n-1}{2}}{k-1}= \mathrm{exp}_2(\binom{n-1}{2}+o(n))$ holds by Stirling's formula. Hence, we have
$$|\Delta(X)|-|X|\geq \mathrm{exp}_2\left(\binom{n-1}{2}+\epsilon n+o(n)\right)- \textrm{exp}_2(\binom{n-1}{2}+o(n)).$$
Thus, if $n$ is sufficiently large given $\epsilon$, $|\Delta(X)|>|X|$. \end{proof}

\bigskip

Now we show that if $X$ is a set of $2$-edge-connected graphs in $\C^{(k)}$, then the number of $2$-edge-connected graphs in the shadow of $X
$ cannot be much less than $|X|$.

\begin{lemma}\label{exp1}
Let $0<\epsilon<1/4$. There exists $n_{3}(\epsilon)$ such that if $n>n_{3}(\epsilon)$, the following holds. Let $M<k<M+n$ and let $X\subset 
\B^{(k)}$. Let $|X|=\binom{x}{k}$ and let $r$ be a positive integer satisfying $r<n$. If $x>\binom{n-(r+1)/2}{2}+\epsilon rn$, then
$$\frac{|\Delta(X)\cap \B^{(k-1)}|}{|X|}>1-\frac{4r}{n^{2}}.$$
\end{lemma}

\begin{proof} Define $U=\Delta(X)\cap\B^{(k-1)}$ and let $B$ be the bipartite graph with vertex partition ($X,U$), the edges being 
the comparable pairs. Every element of $U$ has degree at most $m-k+1$ in $B$. Also, the degree of a graph $G$ in $X$ is exactly $k-|R(G)|$ in 
$B$. Let $a$ be the number of graphs in $X$ with degree at most $k-r-1$ and let $a'$ be the number of graphs in $\B^{(k)}$ with $|R(G)|\geq r
+1$.  Then $a<a'$ and by Lemma \ref{numof2conn}, we have $a'<\binom{\binom{n-(r+1)/2}{2}+\epsilon rn/2}{k}$, provided $n>n_{1}(\epsilon/2)$. 
Moreover, we have the following bounds on the number of edges of $B$:

$$(k-r)(|X|-a')\leq e(B)\leq(m-k+1)|U|.$$
Hence,
$$\frac{|U|}{|X|-a'}\geq\frac{k-r}{m-k+1}.$$
Here, $k\geq m/2+1$, so
$$\frac{|U|}{|X|-a'}\geq \frac{m/2-r+1}{m/2}\geq 1-\frac{4r-4}{n(n-1)}.$$
If $|X|>8n^{3}a'$, we get $\frac{|U|}{|X|}\geq 1-\frac{4r}{n^{2}}$, using that $r\leq n-1$. But note that if $n$ is sufficiently large given 
$\epsilon$, then  $8n^{3}< \mathrm{exp}_2(\epsilon n/3)$, which means that $$8n^{3}a'<\binom{\binom{n-(r+1)/2}{2}+\epsilon rn/2}{k} \cdot 
\mathrm{exp}_2(\epsilon n/3) < \binom{\binom{n-(r+1)/2}{2}+\epsilon rn}{k}<\binom{x}{k},$$
where the second inequality is a consequence of \ref{app2} from Lemma \ref{binom}.  \end{proof}

\bigskip

We remark that we do not have to consider the case when $r\geq n$. If $|X|=\binom{x}{k}\geq 1$, then $x\geq k$, and we can always find $r<n$ satisfying $x>\binom{n-(r+1)/2}{2}+\epsilon rn$. This remark holds true for the upcoming lemmas as well.

In the next lemma, we show that if $X\subset \C^{(k)}$  is a set of non-$2$-edge-connected graphs, then the size of the shadow of $X$ is  
slightly larger than $|X|$.

\begin{lemma}\label{exp2}
Let $\epsilon$ be a positive real number such that $\epsilon<1/2$. There exists $n_{4}(\epsilon)$ such that if $n>n_{4}(\epsilon)$, the 
following holds. Let $k$ be a positive integer with $M<k<M+n$ and let $X\subset \C^{(k)}\setminus\B^{(k)}$. Let $|X|=\binom{x}{k}$ and let $r
$ be a positive integer such that $r<n$ and $x>\binom{n-(r+1)/2}{2}+\epsilon rn$. Then
$$\frac{|\Delta(X)|}{|X|}>1+\frac{4-4r/n}{n}.$$
\end{lemma}

\begin{proof}
Define the bipartite graph $B$ between $X$ and $U=\Delta(X)$ as follows. Let $G\in X$ and $H\in \Delta(X)$ be connected by an edge if $H<G$ 
and $Bt(G)=Bt(H)$. If $T=(C,\{A_{1},...,A_{t}\})$ is the block tree of some graph, let $X(T)$ be the set of graphs in $X$ with 
block tree $T$, and define $U(T)$ similarly. Let $B(T)$ be the bipartite subgraph of $B$ induced on $X(T)\cup U(T)$, and let us estimate $|
U(T)|/|X(T)|$. If $H\in U(T)$ and $e\in [n]^{(2)}\setminus E(H)$ is an edge connecting $A_{i}$ and $A_{j}$ with $i\neq j$, then the block 
tree of $H'=H\cup \{e\}$ differs from $T$. Hence, the degree of $H$ in this bipartite graph is at most
$$u_{T}=m-k+1-\sum_{1\leq i<j\leq t}|A_{i}||A_{j}|+t-1.$$
Note that the term  $t-1$ corresponds to the number of edges in $C$. Now let $G\in X(T)$ and $e\in E(G)$. We have $Bt(G-e)=T$ if and only if $e\in G[A_{i}]\setminus R(G[A_{i}])$ for some $i\in [t]$.  Hence, the degree of $G$ in $B(T)$ is
$$x_{T}(G)=k-\sum_{i=1}^{t}|R(G[A_{i}])|-(t-1).$$

Suppose that $t\geq 3$ or $\min\{|A_{1}|,|A_{2}|\}>1$. Then by Lemma \ref{tech}, we have

$$x_{T}(G)-u_{T}= \sum_{1\leq i<j\leq t}|A_{i}||A_{j}|-2(t-1)-\sum_{i=1}^{t}|R(G[A_{i}])|\geq n.$$
Setting $x_{T}=u_{T}+n$, we have $x_{T}(G)\geq x_{T}$.

Bounding the edges of $B(T)$ in two different ways, we get

$$|X(T)|x_{T}\leq e(B(T))\leq|U(T)|u_{T}.$$
We now consider the remaining case, when $t=2$ and $\min\{|A_{1}|,|A_{2}|\}=1$. Note that we need not consider the case $t=1$ as $T$ is not the block tree of a $2$-edge-connected graph. Without loss of generality, let $|A_{1}|=1$. We have $u_{T}\leq M-(n-2)$, while $x_{T}(G)\geq M-|R(G[A_{2}])|$ for every $G\in X(T)$. Let $a$ be the 
number of graphs $G$ in $X(T)$ with $|R(G[A_{2}])|\geq r+3$. By Lemma \ref{numof2conn}, we have $$a<\binom{\binom{(n-1)-(r+3)/2}{2}+\epsilon 
rn/2}{k},$$ if $n>n_{1}(\epsilon/2)$.

Counting the number of edges of $B(T)$ two ways, we get the following bounds:

$$(|X(T)|-a)(M-(r+2))\leq e(B(T)) \leq (M-(n-2))|U(T)|.$$

Hence,

$$ \frac{|U(T)|}{|X(T)|-a}\geq\frac{M-(r+2)}{M-(n-2)}=1+\frac{(n-r)}{M-(n-2)}>1+\frac{(4n-4r)}{n(n-1)}.$$

If $|X(T)|>2n^{3}a$, this implies $\frac{|U(T)|}{|X(T)|}\geq 1+\frac{4n-4r}{n^{2}-1}.$

Let $\mathcal{T}_{0}$ be the set of pairs $T=(C,\{A_{1},A_{2}\})$ satisfying the following conditions: $T$ is the block tree of some graph in $\C$, $|
A_{1}|=1$, and $|X(T)|\leq 2n^3a$. Let $X_{0}=\bigcup_{T\in \mathcal{T}_{0}}X(T)$. Note that $|\mathcal{T}_{0}|<n^{2}$ as we have at most $n$ choices for $A_{1}$ 
and at most $n-1$ choices for the one edge in $C$. Hence, we have $|X_{0}|\leq 2n^{5}a$. This gives the following bound on the size of $
\Delta(X)$.

$$|\Delta(X)|\geq \left(1+\frac{4n-4r}{n^{2}-1}\right)(|X|-|X_{0}|)\geq $$
$$\geq \left(1+\frac{4n-4r}{n^{2}-1}\right)|X|-4n^{5}\binom{\binom{(n-(r+1)/2}{2}+\epsilon rn/2}{k}.$$

Therefore, if $|X|\geq 4n^{9}\binom{\binom{(n-(r+1)/2}{2}+\epsilon rn/2}{k}$, then
$$\frac{|\Delta(X)|}{|X|}\geq 1+\frac{4n-4r}{n^{2}}.$$
But if $n$ is sufficiently large given $\epsilon$, we have
$$4n^{9}\binom{\binom{(n-(r+1)/2}{2}+\epsilon rn/2}{k}<\binom{\binom{(n-(r+1)/2}{2}+\epsilon rn}{k}\leq |X|.$$
\end{proof}

\bigskip

In the next lemma, we show that if the number of $2$-edge-connected graphs in $X$ is not in the same range as the number of non-$2$-edge-connected graphs in $X$, then $|X|<|\Delta(X)|$.

\begin{lemma}\label{diffsizes}
There exists $n_{5}$ such that if $n>n_{5}$, the following holds. Let $M+1\leq k<M+n$, $X\subset \C^{(k)}$ and $Y=X\cap \B$, $Z=X-Y$. Suppose 
that $|Z|>n|Y|$ or $|Y|>n|Z|$. Then $|\Delta(X)|>|X|$.
\end{lemma}

\begin{proof} If $|X|\geq\binom{m-n/2}{k}$, we are done by Lemma \ref{largeX}. So we can suppose that $|X|<\binom{m-n/2}{k}$.

Firstly, consider the case when $|Y|>n|Z|$. Let $Y=\binom{y}{k}$, then $y<m-n/2$. As $\partial(Y)=\Delta(Y)$, we can apply Lemma \ref{lovasz} 
to get
$$\frac{|\Delta(Y)|}{|Y|}>\frac{k}{m-n/2-k}\geq 1+\frac{2}{n}.$$
Hence, $|\Delta(X)|\geq |\Delta(Y)|>|Y|+2|Y|/n>|Y|+|Z|.$

Now consider the case when $|Z|>n|Y|$. Let $|Z|=\binom{z}{k}$, then $z\geq k=n^{2}/4+O(n)$.

Set $\epsilon=1/40$ and $r=\lceil 2n/3\rceil$. We choose $r$ and $\epsilon$ such that $r<n$ and $z>\binom{n-(r+1)/2}{2}+\epsilon rn$ holds. Hence, by 
Lemma  \ref{exp2}, we have
$$\frac{|\Delta(Z)|}{|Z|}\geq 1+\frac{4-4r/n}{n}\geq 1+\frac{4}{3n},$$
for $n$ sufficiently large. Estimating the size of the shadow of $X$ with $|\Delta(Z)|$, we get

$$|\Delta(X)|\geq |\Delta(Z)|\geq |Z|+\frac{4|Z|}{3n}\geq |Z|+|Y|=|X|.$$

\end{proof}

\bigskip

We also need the following technical lemma, which tells us what conditions need to be satisfied for the sizes of the shadows of $Y,Z$ to have $|X|<|\Delta(X)|$.

\begin{lemma}\label{technical}
Let $a,b,c_{1},c_{2},c_{3}$ be positive real numbers and  $A=a(1+c_{1})$, $B=b(1+c_{2})$ and $C=a(1-c_{3})$. If $c_{3}
\leq c_{1}c_{2}$, then $$a+b\leq C+\max\{B,A-C\}.$$
\end{lemma}

\begin{proof} We need to show that $ac_{3}+b<\max\{B,A-C\}$. Observe that we can suppose that $B=A-C$. Otherwise, if $B<A-C$, we 
can substitute $b$ with $b'>b$ ,and $B$ with $B'=b'(1+c_{3})$, satisfying $B'=A-C$. Then the left hand side of the inequality increases, while the right hand side does not change. We can proceed similarly if $A-C<B$.

If $B=A-C$, then $b=\frac{c_{1}+c_{3}}{1+c_{2}}a$. Hence, our inequality becomes

$$ac_{3}+\frac{c_{1}+c_{3}}{1+c_{2}}a\leq (c_{1}+c_{3})a.$$
Simplifying this inequality, we get that it is equivalent with $c_{3}\leq c_{1}c_{2}.$\end{proof}

\bigskip

Now we are ready to show the existence of a complete matching between the levels close to the middle level.

\begin{theorem}\label{middlelevels}
There exists $n_{6}$ such that if $n>n_{6}$, the following holds. If $M+1\leq k<M+n$, then there exists a complete matching from $\C^{(k)}$ 
to $\C^{(k-1)}$.
\end{theorem}

\begin{proof} By Hall's theorem, it is enough to prove that for any $X\subset \C^{(k)}$, we have $|X|\leq |\Delta(X)|$. Fix $\epsilon=1/18$. 
Let $|X|=\binom{x}{k}$. By Lemma \ref{largeX}, if $x>\binom{n-1}{2}+\epsilon n$, then we are done if $n>n_{2}(\epsilon)$. Now suppose that $x\leq 
\binom{n-1}{2}+\epsilon n$. Let $Y=X\cap \B$ and $Z=X-Y$.
Let $|Y|=\binom{y}{k}$, $|Z|=\binom{z}{k}$, and suppose that $n>n_{5}$. By Lemma \ref{diffsizes}, if  $|Y|>n|Z|$ or $|Z|>n|Y|$, we are done. 
Hence, we can suppose that
$x-\epsilon n<y,z\leq x,$
if $n$ is sufficiently large.

Let $U=\Delta(Y)\cap \B$ and let $r$ be a positive integer satisfying 
$$\binom{n-(r+1)/2}{2}+\epsilon (r+1)n\leq x<\binom{n-r/2}{2}+\epsilon rn$$
One can easily check that as $k\leq x <\binom{n-1}{2}+\epsilon$ and $\epsilon<1/4$, such an $r$  always exists, it is unique, and $r<n$. Furthermore, 
$y,z>\binom{n-(r+1)/2}{2}+\epsilon rn$.

By Lemma \ref{exp1}, if $n>n_{3}(\epsilon)$, we have
$$\frac{|U|}{|Y|}>1-\frac{4r}{n^{2}}.$$
Also, by  Lemma \ref{lovasz}
$$\frac{|\Delta(Y)|}{|Y|}\geq\frac{k}{y-k+1}>\frac{k}{\binom{n-r/2}{2}+2\epsilon rn-k},$$
where the term $2\epsilon rn$ comes from bounding $1+\epsilon rn$ above by $2\epsilon rn$. Using that $k>m/2$, we have

$$\frac{k}{\binom{n-r/2}{2}+2\epsilon rn-k}>\frac{m/2}{\binom{n-r/2}{2}+2\epsilon rn-m/2}=$$
$$=\frac{1}{1-r(2n-1)/n(n-1)+r^{2}/2n(n-1)+8\epsilon r/(n-1)}>$$
$$>\frac{1}{1-2r/n+r^{2}/2n^{2}+9\epsilon r/n},$$
where the last inequality holds if $n$ is sufficiently large.

Finally, by Lemma \ref{exp2}, if $n>n_{4}(\epsilon)$, we have

$$\frac{|\Delta(Z)|}{|Z|}>1+\frac{4-4r/n}{n}.$$

Now we are ready to estimate $|\Delta(X)|$. We have $$\Delta(X)=U \cupdot ((\Delta(Y)\setminus U)\cup\Delta(Z)),$$ where $\cupdot$ denotes 
disjoint union.
Hence,
$$|\Delta(X)|\geq|U|+\max\{|\Delta(Y)|-|U|,|Z|\}.$$

\begin{figure}
\centering
 \includegraphics[scale=1]{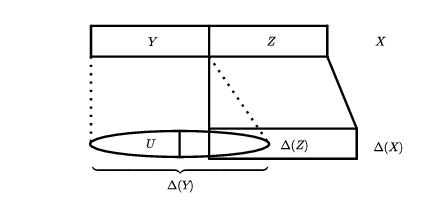}
 \caption{The comparability graph between $X$ and its shadow}
 \label{image2}
 \end{figure}

Also, $|X|=|Y|+|Z|$. Let $c_{1}=\frac{2r/n-r^{2}/2n^{2}-9\epsilon r/n}{1-2r/n+r^{2}/2n^{2}+9\epsilon}$, $c_{2}=\frac{4-4r/n}{n}$ and $c_{3}=4r/n^{2}$. We have  $|\Delta(Y)|>(1+c_{1})|Y|$, $|\Delta(Z)|>(1+c_{2})|Z|$ and $|U|>(1-c_{3})|Y|$. Hence, by Lemma \ref{technical}, 
our task is reduced to proving that $c_{3}\leq c_{1}c_{2}$. Namely,

$$\frac{4r}{n^{2}}\leq \frac{2r/n-r^{2}/2n^{2}-9\epsilon r/n}{1-2r/n+r^{2}/2n^{2}+9\epsilon r/n}\frac{4-4r/n}{n}.$$
Simplifying this inequality, we get
$$1-2r/n+r^{2}/2n^{2}+9\epsilon r/n\leq (2-r/2n-9\epsilon)(1-r/n).$$
For our convenience, let $\alpha=r/n$. Then the previous inequality can be written as
$$1-2\alpha+\alpha^{2}/2+9\epsilon\alpha \leq (2-\alpha/2-9\epsilon)(1-\alpha),$$
which reduces to

$$\alpha+18\epsilon \leq 2.$$
As $\alpha<1$ and $\epsilon=1/18$, this inequality holds. Hence, if $n$ is sufficiently large, we have $|X|\leq|\Delta(X)|$. \end{proof}

\bigskip

We are now ready to prove our main theorem.

\bigskip

\textit{Proof of Theorem \ref{mainthm}.} Let $n>n_{6}$, where $n_{6}$ is the constant given in Lemma \ref{middlelevels}. By Lemma 
\ref{trivial}, it is enough to prove that for $k=1,...,M-1$ there is a complete matching from $\C^{(k)}$ to $\C^{(k+1)}$, and for $k=M
+1,...,m$, there is a complete matching from $\C^{(k)}$ to $\C^{(k-1)}$. But we proved exactly this statement in Lemma \ref{smalllevels}, 
Lemma \ref{largelevels} and Theorem \ref{middlelevels}.
\newline
\hspace*{\fill}\hfill$\Box$

As a final remark, we observe that the proof also shows that $\C^{(M)}$ is the unique largest antichain, as the strict inequality $|\Delta(X)|>|X|$ holds.
\section{Open problems}\label{sect:problems}

In this section, we propose several open problems.

The first problem we propose is inspired by the question investigated in \cite{induced}, which we mentioned in the Introduction. Let $G$ be a 
connected graph and let $C'(G)$ be the family of subgraphs of $G$ that are connected on the vertex set $V(G)$. Define the partial ordering $<$ on $C'(G)$ as usual: $H<H'$ if $E(H)\subset E(H')$.

\begin{question}\label{open:question1}
Let $G$ be a connected graph. Is $(C'(G),<)$ Sperner?
\end{question}

We believe that there should be graphs $G$ for which $(C'(G),<)$ is not Sperner. Unfortunately, even for small graphs, it is difficult to check this property.

We also propose another variation of Question \ref{question}. Let $GP$ be a monotone graph property (a family of graphs closed under isomorphism, and adding edges) and 
let $GP_{n}$ denote the family of graphs in $GP$ with vertex set $[n]$. Also, for $k=0,...,\binom{n}{2}$ let $GP_{n}^{(k)}$ be the set of 
graphs in $GP_{n}$ with $k$ edges. Define the partial ordering $<$ on $GP_{n}$ as usual. The poset $(GP_{n},<)$ might not be graded, however 
it still makes sense to ask the following question. For which graph properties $GP$ is it true that the largest antichain in $(GP_{n},<)$ is 
$GP_{n}^{(k)}$ for some $k$? To ask a more specific question, we propose the following problem.

\begin{question}
Let $H$ be the family of Hamiltonian graphs. Is $(H_{n},<)$ Sperner?
\end{question}

Finally, we suggest the following variation of Question \ref{question}. Suppose we do not distinguish graphs that are isomorphic. More 
precisely, define the equivalence relation $\sim$ on $\C$ such that $G\sim H$ if $G$ and $H$ are isomorphic, and let $\C_{0}$ be the set of 
equivalence classes of $\C$. Define $<$ on $\C_{0}$ such that for $\widetilde{G},\widetilde{H}\in \C_{0}$ we have 
$\widetilde{G}<\widetilde{H}$ if there exists $G\in \widetilde{G}$ and $H\in\widetilde{H}$ satisfying $G<H$ in $(\C,<)$.

\begin{question}\label{open:question2}
Is $(\C_{0},<)$ Sperner?
\end{question}

\section{Acknowledgements}

We would like to thank the anonymous referees for their useful comments and suggestions, and Andrew Thomason for drawing our attention to the simple proof presented in Claim \ref{multigraph}.





\begin{thebibliography}{99}


\bibitem{anderson}
I. Anderson,
\emph{Combinatorics of Finite Sets,}
Oxford University Press (1987).




\bibitem{graphtheory}
Flajolet, Sedgewick,
\emph{Analytic Combinatorics,}
Cambridge University Press (2009).

\bibitem{hall}
P. Hall,
\emph{On Representatives of Subsets,}
J. London Math. Soc., 10 (1) (1935): 26-30.

\bibitem{induced}
M. S. Jacobson, A. E. K\'{e}zdy, S. Seif,
\emph{The poset on connected induced subgraphs of a graph need not be Sperner,}
Order, 12 (3) (1995): 315-318.

\bibitem{katona}
Gy. O. H. Katona,
\emph{Personal communication.}

\bibitem{katona1}
Gy. O. H. Katona,
\emph{A theorem of finite sets},
Theory of Graphs, Akad\'{e}mia Kiad\'{o}, Budapest (1968): 187-207.

\bibitem{kruskal}
J. B. Kruskal,
\emph{The number of simplicies in a complex},
Mathematical Optimization Techniques, Univ. of California Press (1963): 251-278.

\bibitem{lovaszlemma}
 L. Lov\'{a}sz,
 \emph{ Combinatorial Problems and Exercises,}
 North-Holland, Amsterdam (1993).

\bibitem{sperner}
E. Sperner,
\emph{"Ein Satz \"{u}ber Untermengen einer endlichen Menge"},
Mathematische Zeitschrift (in German), 27 (1) (1928): 544-548.

\end{thebibliography}
\end{document}